\documentclass[12pt]{article}

\newcommand{\blind}{1}

\usepackage{mathtools}
\usepackage{tikz}
\usepackage{adjustbox}

\usepackage[round]{natbib}

\usepackage{amsmath, amsfonts, amssymb, amsthm, dsfont}

\usepackage{mathrsfs}

\usepackage{soul}

\usepackage{bm}

\usepackage{booktabs} 

\definecolor{mypurple}{RGB}{0,0,0}
\definecolor{myblue}{RGB}{0,87,120}
\definecolor{aqua}{RGB}{0,0,0}
\definecolor{myorange}{RGB}{0,0,0}
\definecolor{mygrey}{RGB}{255,255,255}

\usepackage{ulem}

\usepackage{graphicx}
\usepackage{color}
\usepackage{nicefrac}
\graphicspath{{Figures/}}

\usepackage{scrextend}

\usepackage{titlesec}
\titleformat{\section}[hang]{\large\center\scshape}{\thesection.}{1em}{}
\titleformat{\subsection}[hang]{\large}{\thesubsection.}{1em}{}
\titleformat{\subsubsection}[hang]{}{\thesubsubsection.}{1em}{}

\usepackage[colorlinks=true, linkcolor=red, urlcolor=blue, citecolor=blue]{hyperref}

\newtheoremstyle{mytheoremstyle} 
    {0.3cm}                      
    {0cm}                        
    {\itshape}                   
    {}                           
    {\scshape}                   
    {: }                          
    {0em}                       
    {}  

\theoremstyle{mytheoremstyle}
\newtheorem{Theorem}{Theorem}

\newtheorem{Lemma}{Lemma}
\newtheorem{Corollary}{Corollary}
\newtheorem{Proposition}{Proposition}

\newtheoremstyle{myExampleRemarkstyle} 
    {0.3cm}                    
    {0cm}                           
    {\itshape}                   
    {}                           
    {\scshape}                   
    {: }                          
    {0em}                       
    {}  

\theoremstyle{myExampleRemarkstyle}

\newtheorem{Assumption}{Assumption}

\renewcommand{\theAssumption}{\Alph{Assumption}}

\newtheorem{innercustomgeneric}{\customgenericname}
\providecommand{\customgenericname}{}
\newcommand{\newcustomtheorem}[2]{%
  \newenvironment{#1}[1]
  {%
   \renewcommand\customgenericname{#2}%
   \renewcommand\theinnercustomgeneric{##1}%
   \innercustomgeneric
  }
  {\endinnercustomgeneric}
}

\newcustomtheorem{customthm}{Theorem}
\newcustomtheorem{customlemma}{Lemma}

\newtheoremstyle{simuStyle}
{0.3cm} 
{0cm} 
{} 
{} 
{\bfseries} 
{.} 
{0em} 
{} 

\theoremstyle{simuStyle}

\newtheoremstyle{stratStyle}
{0.3cm} 
{0cm} 
{} 
{} 
{\scshape} 
{: } 
{0em} 
{} 

\theoremstyle{stratStyle}

\DeclareSymbolFont{lettersA}{U}{txmia}{m}{it}
\DeclareMathSymbol{\real}{\mathord}{lettersA}{"92}
\DeclareMathSymbol{\field}{\mathord}{lettersA}{"83}

\def\real{{\rm I\!R}}

\DeclareMathOperator*{\Int}{Int}

\def\0{{\bf 0}}

\DeclareMathOperator*{\argzero}{argzero}

\def\bt{\bm{\theta}}

\def\N{\mathbb{N}}

\def\boxit#1{\vbox{\hrule\hbox{\vrule\kern3pt
          \vbox{\kern3pt#1\kern3pt}\kern3pt\vrule}\hrule}}

\definecolor{pinegreen}{rgb}{0.0, 0.47, 0.44}

\begin{document}

\def\spacingset#1{\renewcommand{\baselinestretch}%
 {#1}\small\normalsize} \spacingset{1}

\let\refBKP\ref
\renewcommand{\ref}[1]{{\upshape\refBKP{#1}}}

\if1\blind
{
    \title{\bf Asymptotically Optimal Bias Reduction for Parametric Models}
    \author{St\'ephane Guerrier,
        Mucyo Karemera,
        Samuel Orso and\\
        Maria-Pia Victoria-Feser \\
        \vspace{.01cm}\\
        Research Center for Statistics, GSEM, University of Geneva.
    }
     \date{}
    \maketitle
} \fi

\if0\blind
{
    \begin{center}
        {\Large\bf Phase~Transition~Unbiased Estimation \\ \vspace{.2cm}
        in High~Dimensional Settings}
    \end{center}
    \medskip
} \fi

\vspace{-1cm}

\begin{abstract}
An important challenge in statistical analysis concerns the control of the finite sample bias of estimators. This problem is magnified in high-dimensional settings where the number of variables $p$ diverges with the sample size $n$, as well as for nonlinear models and/or models with discrete data. For these complex settings, we propose to use a general simulation-based approach and show that the resulting estimator has a bias of order $\mathcal{O}(0)$, hence providing an asymptotically optimal bias reduction. It is based on an initial estimator that can be slightly asymptotically biased, making the approach very generally applicable. This is particularly relevant when classical estimators, such as the maximum likelihood estimator, can only be (numerically) approximated. We show that the iterative bootstrap of~\citet{kuk1995asymptotically} provides a computationally efficient approach to compute this bias reduced estimator. We illustrate our theoretical results in simulation studies for which we develop new bias reduced estimators for the logistic regression, with and without random effects. These estimators enjoy additional properties such as robustness to data contamination and to the problem of separability.

\vspace{.2cm}
\end{abstract}%

\newpage
\section{Introduction}
\label{eq:intro}
An important challenge in statistical analysis concerns the control of the finite sample bias of estimators. This problem is typically magnified in models with a large number of variables $p$ that are possibly allowed to diverge with the sample size $n$ (see for example \citealp{sur2019modern}). Thereby, bias reduction techniques have been widely studied (a review can, for example, be found in \citealp{kosmidis2014bias}). These bias reductions can be achieved by simulation methods such as the jackknife \citep{Efro:82} or the bootstrap \citep{Efro:79}, by using analytical approximations to the likelihood function (see for example \citealp{BrDaRe:07,BrDa:08}  and the references therein), or, alternatively, by modifications of the estimating equations, as proposed by \citet{Firt:93} and extended for example in~\citet{KoFi:09,KoFi:11,Kosm:14,Kosm:17}. Under appropriate conditions, the resulting bias reductions obtained from these techniques have been shown to achieve (at best) an order of $\mathcal{O}(n^{-2})$.

The main purpose of this paper is to show the existence of a general method that optimally reduces (in an asymptotic sense) the bias of consistent estimators for parametric models. While the formal developments are presented in the next sections, we provide here an intuitive argument for the method we study. Consider a consistent estimator for $\bt_0$, say $\tilde{\bt}$, that is typically \textit{biased} in finite samples, and define the bias $\mathbf{b}(\bt_0, n) \vcentcolon = {\bm{\pi}}(\bt_0, n) - \bt_0,$ where ${\bm{\pi}}(\bt_0, n)\vcentcolon=\mathbb{E}[\tilde{\bt}]$.
In general, the function ${\bm{\pi}}(\bt,n)$ is not available in closed-form but an unbiased estimator, say ${\bm{\pi}}^\ast({\bt}, n)$, is available via simulation-based techniques (and is formally presented further on). 
Similarly, we define a simulation-based version of the bias function $\mathbf{b}(\bt, n)$, as $\mathbf{b}^*(\bt, n) \vcentcolon = {\bm{\pi}}^*(\bt, n) - \bt$. 
In this context, an unbiased estimator based on $\tilde{\bt}$ is therefore
$\check{\bt} \vcentcolon = \tilde{\bt} - \mathbf{b}^*(\bt_0, n)$.
The continuity of the function $\mathbf{b}(\bt, n)$ and the consistency of $\check{\bt}$ suggest $\mathbf{b}^\ast(\bt_0, n) \approx \mathbf{b}^\ast(\check{\bt}, n)$. 
This approximation can be used to write
\begin{equation*}
    \begin{aligned}
    \check{\bt} &=  \tilde{\bt} -  \mathbf{b}^*(\bt_0, n) \approx  \tilde{\bt} -  \mathbf{b}^*(\check{\bt}, n) 
    = \tilde{\bt} - \left\{\check{\bt} + \mathbf{b}^*(\check{\bt}, n) - \check{\bt}\right\}.
    \end{aligned}
\end{equation*}
Since ${\bm{\pi}}^*(\bt, n) = \bt + \mathbf{b}^*(\bt, n)$, $\check{\bt}$ satisfies $\check{\bt}\approx \tilde{\bt} - \left\{{\bm{\pi}}^*(\check{\bt}, n) - \check{\bt}\right\}$. It is therefore approximately a fixed-point of the function 
\begin{equation}
    \bm{T}(\bt, n) \vcentcolon = \tilde{\bt} - \left\{{\bm{\pi}}^*({\bt}, n) - \bt\right\}.
    \label{eq:fixed:point}
\end{equation}
In this paper, we study the properties of the fixed-point of $\bm{T}(\bt, n)$ that we denote by $\hat{\bt}$. While this estimator is not necessarily unbiased, the above heuristic reasoning implies a strong bias reduction ability. Indeed, we demonstrate that $\hat{\bt}$ achieves an \textit{asymptotically optimal} unbiased property. In order to make this statement precise, we first recall that $f(n) = \mathcal{O}\{g(n)\}$, where $f(n)$ and $g(n)$ are real-valued function, implies, by definition, that there exist a $n^{\ast} > 0$ and a $M > 0$ such that $\lvert f(n)\rvert\leq M\lvert g(n)\rvert$, for all $n\geq n^{\ast}$. In this paper, we demonstrate, under some mild conditions, that
\begin{equation}
    \big\lVert \mathbb{E}[\hat{\bt}] - {\bt}_0 \big\rVert_2 = \mathcal{O}(0),
    \label{eq:bias:cor}
\end{equation}
which means that $\big\lVert\mathbb{E}[\hat{\bt}]-{\bt}_0 \big\rVert_2 = 0$, for all $n \geq n^{\ast}$, or equivalently, that $\hat{\bt}$ is unbiased for all $n\geq n^{\ast}$. In other words, the result provided in \eqref{eq:bias:cor} is optimal, as no stronger bias reduction result can be obtained using the \textit{big $\mathcal{O}$ notation} as a way to quantify the (asymptotic) bias of estimators. As will be explained further on, $\hat{\bt}$ belongs to the class of indirect inference estimators (see \citealp{gourieroux1993indirect}) and is directly linked with the Iterative Bootstrap (IB) approach put forward in \citet{kuk1995asymptotically}.

The assumptions needed to achieve \eqref{eq:bias:cor} are weak enough so that the bias reduction method we propose is valid for a wide range of models and estimators and does not rely on model-based analytical transformations. For example, our framework allows to consider estimators that are discontinuous in $\bt$ (as is commonly the case when considering discrete data models) as well as high-dimensional settings in which $p/n \to 0$. Moreover, this bias reduction approach does not necessarily come at the price of an inflated variance and a trade off can be sought between efficiency and computational cost. Indeed, the asymptotic variance of $\hat{\bt}$ can be made arbitrarily close to the one of $\tilde{\bt}$ by improving the ``quality'' (number of simulations) of the approximation ${\bm{\pi}}^*(\bt, n)$ of $\bm{\pi}(\bt, n)$. We also show that the IB algorithm provides a computationally efficient method for computing $\hat{\bt}$. In addition, we demonstrate that the considered bias reduction method preserves its properties also when $\tilde{\bt}$ is inconsistent as long as the asymptotic bias of the latter has a specific form. The approach is therefore applicable in cases where a consistent estimator may be difficult to obtain but a reasonable approximation is available. These results are in line with the observations made by \citet{kuk1995asymptotically}. Indeed, while his primary focus was to construct simulation-based consistent estimators, using the IB, for Generalized Linear Mixed Models (GLMM), he noticed that ``\textit{[...] the method proposed can lead to estimates which are nearly unbiased even for the variance components while the standard errors are only slightly inflated}''.

The paper is organized as follows. In Section \ref{sec:setting}, we present the mathematical setup in which we place our theoretical results. In particular, we present the simulation-based approach and we state and discuss the formal assumptions that are needed to derive the properties of the resulting estimator. These properties are formally stated in Section \ref{sec:main:res} while the proofs can be found in the appendix. In Section \ref{sec:applica}, we apply our approach to derive optimally bias reduced estimators for the logistic regression, with and without random intercept, in high-dimensional settings and possibly with data contamination. These estimators are based on the MLE and an estimator that is robust to data contamination and to the problem of separability. The theoretical results presented in Section \ref{sec:main:res} are in-line with the simulation studies of Section \ref{sec:applica}.

\section{Mathematical Setup}
\label{sec:setting}

Let $\mathbf{X}(\bt_0)\in\real^n$ denote a random sample generated under the model $F_{\bt_0}$ (possibly conditional on a set of fixed covariates), where $\bt_0\in\bm\Theta\subset\real^p$ is the parameter vector we wish to estimate using the consistent estimator $\tilde{\bt}$. In our setting, the dimension of $\bt_0$ is not fixed and is allowed to diverge together with the sample size $n$. Our discussion throughout this paper considers cases where $\tilde{\bt}$ is ``complex'' enough in the sense that it has no closed-form solution and its finite sample bias is unknown.

The bias reduction technique we consider is based on the estimator $\hat{\bt}$, which is defined as the fixed-point of the function $\bm{T}(\bt, n)$ defined in \eqref{eq:fixed:point}. More precisely, $\hat{\bt}$ can be expressed as follows:
\begin{equation}
\hat{\bt} \vcentcolon = \argzero_{\bt\in\bm\Theta}\;\bt-\bm{T}(\bt, n) =\argzero_{\bt\in\bm\Theta}\;
      \tilde{\bt}-{\bm{\pi}}^*({\bt}, n).
    \label{eq:indirectInf:hTimesN}
\end{equation}
The estimator $\hat{\bt}$ can be seen as a special case of an indirect inference estimator (\citealp{gourieroux1993indirect}) when using for $\bm{\pi}(\bt, n)$ the following approximation
\begin{equation*}
    {\bm{\pi}}^\ast({\bt}, n) \vcentcolon = \frac{1}{H} \sum_{h = 1}^H  \tilde{\bt}^\ast_h,
\end{equation*}
where $\tilde{\bt}^*_h$ is the same estimator as $\tilde{\bt}$ but computed on the \textit{simulated} sample $\mathbf{X}_h^*(\bt)\in\real^n$ under model $F_{\bt}$. We use the subscript $h = 1,\ldots,H$ to identify distinct samples\footnote{The random numbers (which are independent of $\bt$) used to construct $\mathbf{X}_h^\ast(\bt)$ are redrawn with the same seed values. Therefore, when $\tilde{\bt}^*_h$ is computed twice on simulated data from the same model $F_{\bt}$, one gets the same value.}. Moreover, the IB of \citet{kuk1995asymptotically} is also directly related to $\hat{\bt}$. Indeed, the IB algorithm (sequence) can be generally defined as
\begin{equation}
    \bt^{(k+1)} \vcentcolon = \bm{T}\left(\bt^{(k)}, n\right),
    \label{eq:ib:seq}
\end{equation}
with $\bt^{(0)} = \tilde{\bt}$. Under appropriate conditions (discussed further on), the fixed-point $\hat{\bt}$ corresponds to the limit (in $k$) of \eqref{eq:ib:seq} and therefore $\hat{\bt}$ can also be characterized as 
\begin{equation*}
    \hat{\bt} = \lim_{k\to\infty}\;\bt^{(k)}.
\end{equation*}
Therefore, the IB approach can provide a natural and computational efficient algorithm to compute $\hat{\bt}$, or more generally, indirect inference estimators (see \citealp{guerrier2018simulation} for more details). 

In what follows, we state and discuss the conditions allowing $\hat{\bt}$ to satisfy the optimal bias reduction property defined in~\eqref{eq:bias:cor}. Henceforth, we call $\hat{\bt}$ the Optimally Bias REduced Estimator (OBREE).

\setcounter{Assumption}{0}
\renewcommand{\theHAssumption}{otherAssumption\theAssumption}
\renewcommand\theAssumption{\Alph{Assumption}}
\begin{Assumption}
\label{assum:A}
The set $\bm\Theta$ is a compact subset of $\real^p$ and $\bt_0\in\Int(\bm\Theta)$.
\end{Assumption}

Assumption \ref{assum:A} is very mild and is typically used in most settings where estimators have no closed-form. Indeed, the compactness of $\bt$ and $\bt_0 \in \Int(\bm\Theta)$ are common regularity conditions for the consistency and asymptotic normality of the MLE, respectively (see for example \citealp{newey1994large}). Hence, Assumption~\ref{assum:A} may be redundant depending on the requirements already brought by taking $\tilde{\bt}$ to be consistent (and possibly asymptotic normally distributed).

\renewcommand{\theHAssumption}{otherAssumption\theAssumption}
\renewcommand\theAssumption{\Alph{Assumption}}
\begin{Assumption}
\label{assum:B}
The bias function $\mathbf{b}(\bt, n)$ exists and is once continuously differentiable in $\bt \in \bm\Theta$. Moreover, there exists a $\beta > 0$ such that $\lVert\mathbf{b}(\bt, n)\rVert_\infty = \mathcal{O}(n^{-\beta})$ and $p = o(n^{2\beta})$.
\end{Assumption}

Assumption \ref{assum:B} is likely to be satisfied in the majority of practical situations. Indeed, the bias function is typically assumed (at least implicitly) to have certain degree of smoothness (see for example \citealp{kosmidis2014bias}). Such property is expected to hold even in situations where the initial estimator $\tilde{\bt}$ may not be continuous (see Assumptions \ref{assum:C} and \ref{assum:D} for more details), as it is the case, for example, with discrete data models. Moreover, the second part of Assumption \ref{assum:B} is very mild. Indeed, many common estimators, including the MLE, can be expanded in decreasing powers of $n$ such that $\beta = 1$. In these cases, this requirement is always satisfied (since in our case $p \leq n$) and may be suitable in high-dimensional settings where $p/n \to c \in [0, 1)$. Assumption \ref{assum:B} is also particularly useful as it allows to decompose the estimator $\tilde{\bt}$ into a non-stochastic component $\bm{\pi}\left(\bt_0, n\right)$ and a random term $\mathbf{v} \left(\bt_0,n\right)$. Indeed, 
we can write:
\begin{equation}
  \tilde{\bt}  = \bm{\pi} \left( \bt_0, n\right) + \mathbf{v} \left(
    \bt_0,n\right), 
   \label{bias:estim}
\end{equation}
where $\mathbf{v} \left(\bt_0,n\right) \vcentcolon = \tilde{\bt} - \bm{\pi} \left( \bt_0, n\right)$ is a zero-mean random vector. In our next assumptions, which are not required to establish the optimal bias reduction defined in \eqref{eq:bias:cor}, we impose additional restrictions on $\tilde{\bt}$. These requirements will be used to prove the consistency and the asymptotic normality of $\hat{\bt}$. 

\renewcommand{\theHAssumption}{otherAssumption\theAssumption}
\renewcommand\theAssumption{\Alph{Assumption}}
\begin{Assumption}
\label{assum:C}
The variance of $\mathbf{v} \left(\bt,n \right)$ exists and is finite for any $\bt \in \bm\Theta$. Moreover, there exists a $\alpha >0$ such that $\lVert\mathbf{v} \left(\bt,n\right)\rVert_\infty = \mathcal{O}_{\rm p}(n^{-\alpha})$ and $p = o(n^{2\alpha})$.
\end{Assumption}

Assumption~\ref{assum:C} is frequently employed and typically very mild. In the common situation where $\tilde{\bt}$ is \mbox{$\sqrt{n}$-consistent}, then we would have $\alpha = \nicefrac{1}{2}$ and Assumption~\ref{assum:C} would simply require that $p/n \to 0$. In our last assumption, we consider the limiting distribution of $\tilde{\bt}$ (and therefore of $\mathbf{v}(\bt, n)$). The latter is used to derive the asymptotic normality of $\hat{\bt}$, in order to evaluate, in particular, the potential efficiency loss from $\tilde{\bt}$ to $\hat{\bt}$. 

\renewcommand{\theHAssumption}{otherAssumption\theAssumption}
\renewcommand\theAssumption{\Alph{Assumption}}
\begin{Assumption}
\label{assum:D} 
For any $\mathbf{s} \in \real^p$ such that $\lVert\mathbf{s}\rVert_2 = 1$ we have
\begin{equation*}
    \sqrt{n} \mathbf{s}^{T}\bm{\Sigma}(\bt_0)^{-\nicefrac{1}{2}}\left(\tilde{\bt} - \bt_0\right) \xrightarrow{\;d\;} \mathcal{N}\left(\mathbf{0}, 1\right),
\end{equation*}
where $\bm{\Sigma}(\bt)$ is nonsingular and continuous in $\bt$.
\end{Assumption}

While Assumption \ref{assum:D} is frequently used, it may not necessarily be mild. In low-dimensional settings, this assumption is satisfied for a vast majority of commonly used estimators. However, in high-dimensional settings, the validity of the asymptotic normality of $\tilde{\bt}$, as expressed in Assumption \ref{assum:D}, is often unknown for many models. As previously mentioned, this assumption allows to remove the requirements on $\alpha$ and $\beta$ of Assumptions \ref{assum:A} and \ref{assum:B} provided that $p/n \to 0$.

Finally, our assumption framework is not necessarily the weakest possible in theory and may be further relaxed. However, we do not attempt to pursue the weakest possible conditions to avoid overly technical treatments in establishing the theoretical results presented in the following section.

\section{Main Results}
\label{sec:main:res}
In this section, we present the main properties of the OBREE $\hat{\bt}$ under the assumptions presented in Section \ref{sec:setting}. Our results are valid for all $H \geq 1$ and in high-dimensional settings and, in most cases, are applicable when $p/n \to 0$. Theorem \ref{thm:bias} shows that $\hat{\bt}$ is an optimally bias reduced estimator in the sense of \eqref{eq:bias:cor}. %
\begin{Theorem}
\label{thm:bias}
Under Assumptions \ref{assum:A} and \ref{assum:B}, the estimator $\hat{\bt}$ satisfies
\begin{equation*}
    \big\lVert\mathbb{E}[\hat{\bt}]-{\bt}_0\big\rVert_2=\mathcal{O}(0).
\end{equation*}
\end{Theorem}
The proof of Theorem \ref{thm:bias} is given in Appendix \ref{app:unbias:IB}. This result uses a technical lemma (also presented in Appendix \ref{app:unbias:IB}) which provides a general strategy for proving that a function is $\mathcal{O}(0)$. This result is arguably a major improvement on bias reduction techniques as it provides an optimal asymptotic reduction while not relying on model-based analytical transformations. This approach is, therefore, widely and readily applicable. Interestingly, the result of Theorem~\ref{thm:bias} is, in some cases, valid in high-dimensional settings where $p/n \to c \in [0, 1)$. An equivalent statement of Theorem~\ref{thm:bias} is that there exists a finite sample size $n^\ast$ such that $\lVert\mathbb{E}[\hat{\bt}] - \bt_0\rVert_2 = 0$ for all $n$ greater than $n^\ast$. In our experience, as illustrated with the simulation studies presented in Section \ref{sec:applica}, the value $n^\ast$ seems to be reasonably small as the OBREE does indeed appear to be unbiased with relatively small sample sizes.

In addition, the result of Theorem~\ref{thm:bias} may still hold in the wider situation where $\tilde{\bt}$ is not consistent. While a thorough investigation is left for further research, a preliminary result can be found in Corollary \ref{coro:thm:bias} in Appendix \ref{app:unbias:IB} where $\tilde{\bt}$ is assumed to have a sub-linear asymptotic bias. As it will be illustrated in the simulation studies of Section~\ref{sec:applica}, we find that when the asymptotic bias is ``small'', the OBREE appears to preserve its finite sample bias properties. The latter is therefore applicable in cases where a consistent estimator may be difficult to obtain for example for computational (and/or numerical) reasons, but a ``close'' approximation is available.

The next result concerns the standard statistical properties of the OBREE, namely consistency and asymptotic normality. While, in our framework, these properties are trivially satisfied in low-dimensional settings, obtaining them in high-dimensional settings may be more challenging.  

\begin{Proposition}
\label{thm:stat_prop}
Under Assumptions \ref{assum:A}, \ref{assum:B} and \ref{assum:C}, the OBREE is such that 
\begin{equation*}
    \big\lVert \hat{\bt} - \bt_0 \big\rVert_2 = o_{\rm p}(1).
\end{equation*}
Moreover, with the addition of Assumption \ref{assum:D}, for any $\mathbf{s} \in \real^p$ such that $\lVert\mathbf{s}\rVert_2 = 1$ the OBREE satisfies
\begin{equation*}
    \sqrt{n} \mathbf{s}^{T}\left\{\left(1 + \frac{1}{H}\right)\bm{\Sigma}(\bt_0)\right\}^{-\nicefrac{1}{2}}\left(\hat{\bt} - \bt_0\right) \xrightarrow{\;d\;} \mathcal{N}\left(\mathbf{0}, 1\right).
\end{equation*}
\end{Proposition}
The proof of Proposition~\ref{thm:stat_prop} is given in Appendix~\ref{app:stat_prop}. This result shows that the efficiency of the OBREE $\hat{\bt}$ (relative to the initial estimator $\tilde{\bt}$) can be made arbitrarily close to one by increasing the number of simulations $H$. Therefore, this highlights that a trade-off between efficiency and computational cost can be made. Interestingly, the finite sample variance of the OBREE can be smaller than the one of the initial estimator as, for example, illustrated in the simulation studies presented in Section~\ref{sec:applica}. In low-dimensional settings, Proposition \ref{thm:stat_prop} is implied by the results on indirect inference estimators presented, for example, in~\citet{gourieroux1996simulation}.

To compute the OBREE, a computationally efficient approach is the IB algorithm provided in~\eqref{eq:ib:seq}. Indeed, Proposition~\ref{prop:ib} shows that the IB sequence converges (exponentially fast) to the OBREE $\hat{\bt}$.
\begin{Proposition}
	\label{prop:ib}
		Under Assumptions \ref{assum:A}, \ref{assum:B} and \ref{assum:C}, the IB sequence satisfies
         \begin{equation*}
         	\left\lVert\hat{\bt}^{(k)} - \hat{\bt}\right\rVert_2  =o_{\rm p}\left\{\exp(-k)\right\}.
         \end{equation*}
\end{Proposition}
The proof of Proposition \ref{prop:ib} is given in Appendix \ref{app:ib}. Proposition \ref{prop:ib} shows that the IB algorithm converges to $\hat{\bt}$ (in norm) at an exponential rate. In our practical experience, the number of iterations necessary to reach a suitable neighbourhood of the solution appears to be relatively small (typically less than 20 iterations).

The results presented in this section can be used in a wide range of practical settings/models, to obtain estimators with {\it optimal} finite sample properties. In Section \ref{sec:applica} below, we provide the results of simulation studies involving models for binary outcomes, with and without random effect, and with different initial estimators. For one setting, a comparison can be made with an available bias reduced estimator \citep{KoFi:09}, and for the others, our results provide new (optimally) bias reduced estimators.

\section{Application to Binary Response Models}
\label{sec:applica}
In this section, we apply the methodology developed in Sections \ref{sec:setting} and \ref{sec:main:res} to investigate the performance of the OBREE. First, we consider the logistic regression \citep{NeWe:72,McCuNe:89}, for which the MLE is known to be biased. It is one of the most commonly used model for binary responses and several bias reduction methods have been proposed as for example the bias reduced estimator of \citet{KoFi:09}. To illustrate the flexibility of the OBREE, we choose two different initial estimator $\tilde{\bt}$, the MLE and a robust estimator. We compare these OBREEs to their initial estimators as well as to the bias reduced estimator proposed by \citet{KoFi:09}. Their respective performance is studied when the data is generated at the model but also under slight model misspecification where outliers are randomly created. Then, we extend the logistic regression to include a random intercept, a special case of GLMM \citep[see for example][]{LeNe:01,McCuSe:01,JiangBook2007}, for which there is no closed-form expression for the likelihood function. In this case, the initial estimator is selected to be a numerically simple approximation of the MLE, and its finite sample behaviour is compared to the ones of the MLE computed using several more precise approximation methods.

\subsection{Bias Reduced Estimators for the Logistic Regression}
\label{sec:logistic}
It is well known that in some quite frequent practical situations, the MLE of the logistic regression is biased and/or its computation can become very unstable, especially when performing some type of resampling scheme for inference. The underlying reasons are diverse, but the main ones are the possibly large $p/n$ ratio, separability (leading to regression slope estimates of infinite value) and data contamination (robustness). The first two sources are often confounded and practical solutions are continuously sought to overcome the difficulty in performing ``reasonable'' inference. For example, in medical studies, the bias of the MLE together with the problem of separability has led to a rule of thumb called the number of Events Per Variable (EPV), that is the number of occurrences of the least frequent event over the number of covariates, which is used in practice to choose the maximal number of covariates one is ``allowed'' to use in a logistic regression (see for example \citealp{AuSt:17} and the references therein). Moreover, the MLE is known to be sensitive to slight model deviations that take the form of outliers in the data, leading to the proposal of several robust estimators for the logistic regression and more generally for Generalized Linear Model (GLM) (see for example \citealp{CaRo:01b,Cize:08,HeCaCoVF:09} and the references therein).

Despite all the available estimators, to the best of our knowledge, none is able to handle the three potential sources of bias jointly, namely small EPV (high-dimensional settings), separation and data contamination. In this section, we make use of the OBREE, which is built through a simple adaptation of available estimators. Although our choices for the initial estimators may not be the optimal ones for this problem, they nevertheless provide OBREE with advantageous properties. Indeed, they appear to be unbiased and have similar finite sample Mean Squared Error (MSE) as the bias reduced MLE of \citet{KoFi:09} in uncontaminated data settings. Additionally, in data contaminated settings, the performance of the OBREE based on the robust initial estimator remains nearly unchanged. Moreover, in the latter case, we adapt the initial (robust) estimator so that it is not affected by the problem of separability, a problem that is more severe in the case of robust estimators \citep[see][]{Rousseeuw2003}.

Consider the logistic regression with response $\mathbf{Y} \vcentcolon = \mathbf{Y}(\bm{\beta}_0)\in \{0, 1\}^n$ and linear predictor $\mathbf{X}\boldsymbol{\beta}$, where $\mathbf{X}$ is an $n\times p$ matrix of fixed covariates with rows $\mathbf{x}_i,i=1,\ldots,n$, and with logit link ${\mu}_i(\bm{\beta})\vcentcolon = \mathbb{E}[\mathbf{Y}_i]  =\exp(\mathbf{x}_i\boldsymbol{\beta})/\{1+\exp(\mathbf{x}_i\boldsymbol{\beta})\}$. The MLE for $\bm{\beta}$ is given by 
\begin{equation}
        \tilde{\bm{\beta}} \vcentcolon = \argzero_{\bm{\beta} \in \real^p}  \frac{1}{n}\sum_{i=1}^n\mathbf{x}_i\left\{\mathbf{Y}_i-\mu_i(\bm{\beta})\right\}, 
    \label{Eq_MLE-logistic}
\end{equation}
and can be used as an initial estimator in order to obtain the OBREE, using for example the IB algorithm. When using, as initial estimator, the MLE defined in \eqref{Eq_MLE-logistic}, we denote the resulting estimator as the OBREE-MLE.

We also consider the robust $M$-estimator proposed by \citet{CaRo:01b}, with general estimating equations (for GLMs) given by
\begin{equation}
    \bm{\psi}\left(\bm{\beta}, \mathbf{Y}_i\right)\vcentcolon = \psi_c\left\{r\left(\bm{\beta}, \mathbf{Y}_i\right)\right\}w\left(\mathbf{x}_i\right)V^{-1/2}\left\{\mu_i(\bm{\beta})\right\}(\partial/\partial\bm{\beta})\mu_i(\bm{\beta})-\mathbf{a}\left(\bm{\beta}\right),
    \label{Eq_rob-glm}
\end{equation}
with $r\left(\bm{\beta}, \mathbf{Y}_i\right) \vcentcolon = \left\{\mathbf{Y}_i-\mu_i(\bm{\beta})\right\}V^{-1/2}\left\{\mu_i(\bm{\beta})\right\}$ being the Pearson residuals and with consistency correction factor
\begin{equation}
    \mathbf{a}\left(\bm{\beta}\right) \vcentcolon = \frac{1}{n}\sum_{i=1}^n\mathbb{E}\left[\psi_c\left\{r\left(\bm{\beta}, \mathbf{Y}_i\right)\right\}w\left(\mathbf{x}_i\right)V^{-1/2}\left\{\mu_i(\bm{\beta})\right\}(\partial/\partial\bm{\beta})\mu_i(\bm{\beta})\right],
    \label{Eq_rob-consist}
\end{equation}
where the expectation is taken over the (conditional) distribution of the responses $\mathbf{Y}_i$ (given $\mathbf{x}_i$). For the logistic regression, we have  $V\left\{\mu_i(\bm{\beta})\right\} \vcentcolon = \mu_i(\bm{\beta})\{1-\mu_i(\bm{\beta})\}$. To avoid a potential problem of separatibility, we follow the suggestion of~\citet{Rousseeuw2003} and compute the robust initial estimator on the transformed responses, or \textit{pseudo-values}:
\begin{equation}
    \widetilde{\mathbf{Y}}_i \vcentcolon = (1-\delta)\mathbf{Y}_i +\delta \left(1-\mathbf{Y}_i\right), \;\; \forall i=1,\dots,n,
    \label{Eq_pseudo-val}
\end{equation}
where $\delta \in \left[0,0.5\right)$ is a fixed scalar close to zero. Using the pseudo-values leads to an initial (robust) estimator that is not consistent, even if we expect the asymptotic bias to be very small. The finite sample performance of this OBREE is of interest to investigate to what extent, the conditions needed for its bias property to hold, can be enlarged (as hinted in Corollary \ref{coro:thm:bias}). To compute the initial estimator, we use the implementation of the \texttt{glmrob} function provided in the \texttt{robustbase} package in R \citep{robustbase2018}, with $\psi_c$ in~\eqref{Eq_rob-glm} being the Huber loss function (with default parameter $c$, see \citealp{huber1964robust}) and $w\left(\mathbf{x}_i\right)=\sqrt{1-h_{ii}}$, $h_{ii}$ being the $i$th diagonal element of the hat matrix $\mathbf{X}\left(\mathbf{X}^T\mathbf{X}\right)^{-1}\mathbf{X}^T$. The resulting estimator, called OBREE-R, is in fact robust in the sense that it has a bounded influence function \citep{Hamp:74}. 
We perform a simulation study to validate the properties of the OBREE-MLE and OBREE-R and compare their finite sample performances to other well established estimators. In particular, as a benchmark, we also compute the MLE, the bias reduced MLE (BR-MLE) using the \texttt{brglm} function (with default parameters) of the \texttt{brglm} package in R \citep{brglm2017}, as well as the robust estimator (\ref{Eq_rob-glm}) using the \texttt{glmrob} function in R without data transformation (ROB). We consider four situations that can occur with real data, which result from the combinations of balanced outcome classes (Setting I) and unbalanced outcome classes (Setting II) with and without data contamination. We also consider a large model with $p=200$ and choose $n$ as to provide EPV of respectively $5$ and $3.75$, which are below the usually recommended value of $10$. The parameter values for the simulations are provided in Table~\ref{tab:sim-logistic}. 

 \begin{table}[!hb]
     \centering
     \caption{Simulation settings for the logistic regression.}
     \begin{tabular}{lrr}
 \toprule
 Parameters &  Setting I & Setting II \\
 \midrule
 $p=$ & $200$ & $200$ \\
 $n=$ & $2000$ & $3000$ \\
 $\sum_{i=1}^ny_i\approx$ & $1000$ & $750$ \\
 EPV $\approx$ & 5 & 3.75 \\
 $H=$ & $500$ & $500$ \\
 $\beta_1=\beta_2=$ & $5$ & $5$ \\
 $\beta_3=\beta_4=$ & $-7$ & $-7$ \\
 $\beta_5=\ldots=\beta_{200}=$ & $0$ & $0$ \\
 $\delta=$ & $0.01$ & $0.01$ \\
 Simulations & $1000$ & $1000$ \\
 \bottomrule
     \end{tabular}
     \label{tab:sim-logistic}
 \end{table}

The covariates were simulated (similarly to \citealp{CaSu:19}) independently from distributions $\mathcal{N}(0,4/\sqrt{n})$ for Setting I and $\mathcal{N}(0.6,4/\sqrt{n})$ for Setting II, in order to ensure that the size of the log-odds ratio $\mathbf{x}_i\bm{\beta}$ does not increase with $n$, so that $\mu_i(\bm{\beta})$ is not trivially equal to either $0$ or $1$. To contaminate the data, we choose a misclassification error that allows to observe a noticeable effect on the different estimators, which consists in permuting 2\% of the responses with corresponding larger (smaller) fitted probabilities (expectations). The simulation results are presented in Figure~\ref{fig:sim-logistic-bxp} as boxplots of the finite sample distribution, and in Figure~\ref{fig:sim-logistic-summary} as the bias and Root Mean Squared Error (RMSE) of the different estimators.

 \begin{figure}[!ht]
     \centering
     \includegraphics[width=14cm]{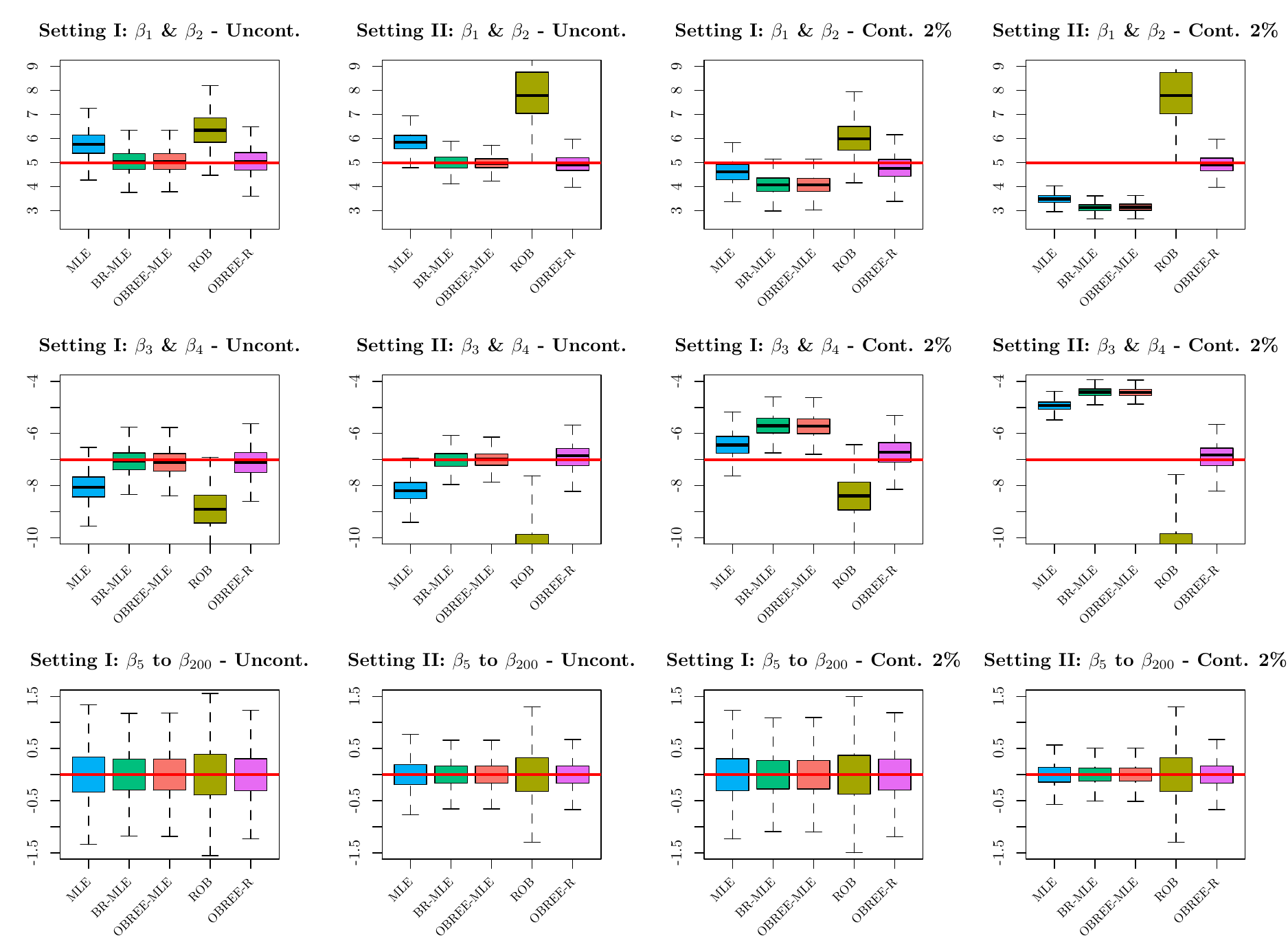}
     \caption{Finite sample distribution of estimators for the logistic regression using the simulation settings presented in Table \ref{tab:sim-logistic}. The estimators are the MLE (MLE), the Firth's bias reduced MLE (BR-MLE), the OBREE based on the MLE as initial estimator (OBREE-MLE), the robust estimator in \eqref{Eq_rob-glm} (ROB) and the OBREE with the robust estimator computed on the pseudo values \eqref{Eq_pseudo-val} as initial estimator (OBREE-R). For each simulation setting, $1000$ samples are generated.}
     \label{fig:sim-logistic-bxp}
 \end{figure}
\begin{figure}[!ht] 
    \centering
    \includegraphics[width=14cm]{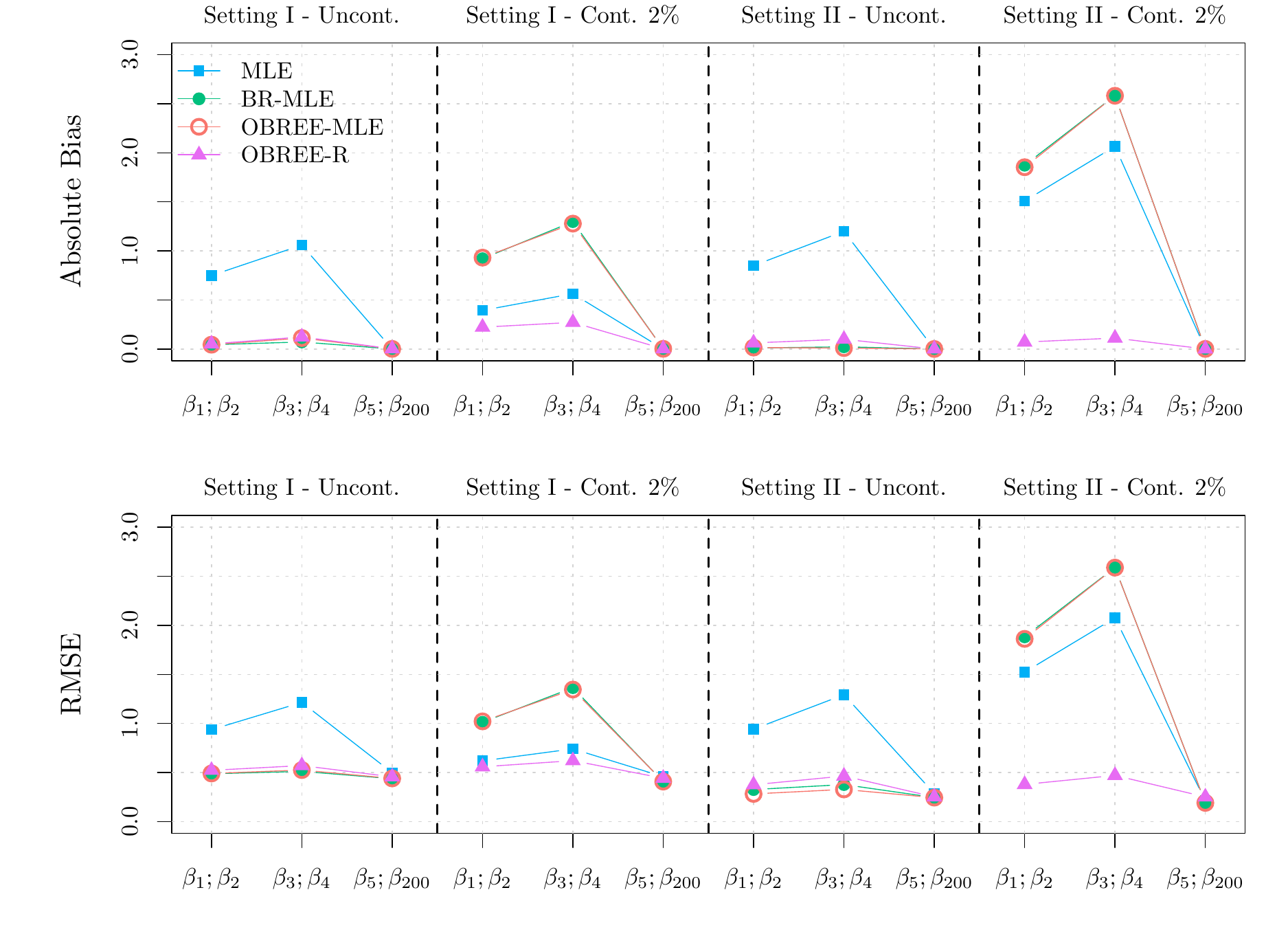}
    \caption{Finite sample bias and RMSE of estimators for the logistic regression using the simulation settings presented in Table \ref{tab:sim-logistic}. The estimators are the MLE (MLE), the Firth's bias reduced MLE (BR-MLE), the OBREE based on the MLE as initial estimator (OBREE-MLE) and the OBREE with the robust estimator in~\eqref{Eq_rob-glm} computed on the pseudo-values (\ref{Eq_pseudo-val}) as initial estimator (OBREE-R). Since the bias and RMSE of the robust estimator in~\eqref{Eq_rob-glm} (ROB) are much larger than the others, we omit them to avoid an unsuitable scaling of the graphs. For each simulation setting, $1000$ samples are generated.}
    \label{fig:sim-logistic-summary}
 \end{figure}
The finite sample distributions presented in Figure~\ref{fig:sim-logistic-bxp}, as well as the summary statistics given by the bias and RMSE presented in Figure~\ref{fig:sim-logistic-summary}, allow us to draw the following conclusions that support the theoretical results. In the uncontaminated case, the MLE and the robust estimator ROB are biased (except when the slope parameters are zero), however, the BR-MLE, OBREE-MLE and OBREE-R are all apparently unbiased. The results for the OBREE-MLE are in-line with its theoretical properties (in particular Theorem~\ref{thm:bias}). The OBREE-R appears to enjoy the same properties although its initial estimator has a ``small'' asymptotic bias (see Corollary \ref{coro:thm:bias}). This illustrates that our approach is applicable even when a consistent initial estimator is not available or when it is numerically unreliable (as is the case here). Moreover, the variability of all estimators is comparable, except for ROB which makes it rather inefficient in these settings. With 2\% of contaminated data (missclassification error), the only estimator whose behaviour remains stable compared to the uncontaminated data setting is the OBREE-R. This is in line with a desirable property of robust estimators, that is stability with or without (slight) data contamination. The behaviour of all estimators remains the same in both settings, that is, whether or not the responses are balanced. Finally, as argued above, a better proposal for a robust, bias reduced and consistent estimator, as an alternative to OBREE-R, could in principle be proposed, but this is left for further research. 

\subsection{Bias Reduced Estimator for the Random Intercept Logistic Regression}
\label{sec:Mlogistic}
An interesting way of extending the logistic regression to account for the dependence structure between the observed responses is to use the GLMM family \citep[see for example][and the references therein]{LeNe:01,McCuSe:01,JiangBook2007}. We consider here a commonly used model in practical settings, namely the random intercept model. The binary response is denoted by $\mathbf{Y}_{ij}\vcentcolon=\mathbf{Y}_{ij}(\bm{\beta}_0, n)$, where $i=1,\ldots,m$ and $j=1,\ldots,n_i$. The expected value of the response is expressed as
\begin{equation}
    \boldsymbol{\mu}_{ij}(\bm{\beta}\vert U_i) \vcentcolon = \mathbb{E}\left[\mathbf{Y}_{ij}\vert U_i\right] = \frac{\exp{\left( \mathbf{x}_{ij}^T\boldsymbol{\beta}+U_i\right)}}{1+\exp{\left(\mathbf{x}_{ij}^T\boldsymbol{\beta}+U_i\right)}},
    \label{Eq_linpredGLMM}
\end{equation}
where $\mathbf{x}_{ij}$ is a $q$-vector of covariates (possibly accounting for a fixed intercept),
$\boldsymbol{\beta}$ is a $q$-vector of regression coefficients and the random effect  $U_i,\; i = 1, ..., m$ is a normal random variable with zero mean and (unknown) variance $\sigma^2$.

Because the random effects are not observed, the MLE is derived on the marginal likelihood function where the random effects are integrated out. These integrals have no known closed-form solutions, so approximations to the (marginal) likelihood function have been proposed, including the 
Penalized Quasi-Likelihood (PQL) \citep[see for example][]{BrCl:93}, Laplace Approximations (LA) \citep[see for example][]{RaYaYo:00} and adaptive Gauss-Hermite Quadrature (GHQ) \citep[see for example][]{PiCh:06}. It is well known that PQL methods lead to biased estimators while LA and GHQ are more accurate \citep[for extensive accounts of methods across software and packages, see for example][]{BOLKER2009127,KiChEm:13}. The \texttt{lme4} R package \citep{BaMaBo:10} uses both the LA and GHQ to compute the likelihood function, while the \texttt{glmmPQL} R function of the \texttt{MASS} library (\citealp{VeRi:02}) uses the PQL.

\begin{table}[!hb]
     \centering
     \caption{Simulation settings for the logistic regression with a random intercept.}
     \begin{tabular}{lrr}
 \toprule
 Parameters &  Setting I & Setting II \\
 \midrule
 $p = q + 1=$ & $31$ & $31$ \\
 $m=$ & $5$ & $50$ \\
 $\forall i,\;n_i=n_\circ =$ & $50$ & $5$\\
 $n=$ & $250$ & $250$ \\
 $\sum_{i=1}^m\sum_{j=1}^{n_i} y_{i,j}\approx$ & $125$ & $125$   \\
 EPV $\approx$ & $4$ & $4$ \\
 $H=$ & $200$ & $200$ \\
 $\beta_0=$ & $0$ & $0$ \\
 $\beta_1=\beta_2=$ & $5$ & $5$ \\
 $\beta_3=\beta_4=$ & $-7$ & $-7$ \\
 $\beta_5=\ldots=\beta_{30}=$ & $0$ & $0$ \\
 $\sigma^2=$ & $1.5$ & $1.5$ \\
 Simulations & $1000$ & $1000$ \\
 \bottomrule 
     \end{tabular}
     \label{tab:sim-logistic2}
 \end{table}

In this section, we consider a simple approximation of the MLE as initial estimator for the OBREE. For computational efficiency, we choose the estimator defined through penalized iteratively reweighted least squares (P-IRLS) (see for example \citealp{BaMaBo:10}) as implemented in the function \texttt{glmer} function (with argument \texttt{nAGQ} set to 0) of the \texttt{lme4} R package. 

To study the behaviour of the OBREE and compare its performance in terms of bias and variance in finite samples to different approximations of the MLE (LA, GHQ or PQL), we perform a simulation study using the two settings described in Table~\ref{tab:sim-logistic2}. Both settings can be considered as high-dimensional in the sense that $p$ is relatively large with respect to $m$. Moreover, while Setting II ($m=50$, $n_i=n_\circ = 5, \forall i$) reflects a possibly more frequent situation, Setting I concerns the case where $m$ is small (and much smaller than $n_\circ$), a situation frequently encountered in cluster randomised trials \citep[see for example][and the references therein]{CRT-Huang-2016,CRT-Leyrat-2017}. As for the simulation study in Section~\ref{sec:logistic} on the logistic regression, the covariates are simulated independently from the distribution $\mathcal{N}(0,4/\sqrt{n})$.   

The finite sample distributions are illustrated in Figure~\ref{fig:sim-glmm-boxplot}.
Figure~\ref{fig:sim-glmm-summary} presents the finite sample bias and RMSE of the approximated MLE estimators and the OBREE. It can be observed that the proposed OBREE has a drastically reduced finite sample estimation bias, especially for the random effect variance estimator. Moreover, the OBREE also achieves the lowest RMSE. These simulation results are in-line with the theoretical properties of the OBREE and the simulation-based findings of~\citet{kuk1995asymptotically}. This has important advantages when performing inference in practice, and/or when the parameter estimates, such as the random intercept variance estimate, are used, for example, to evaluate the sample size needed in subsequent randomized trials. The OBREE can eventually be based on another initial estimators in order to further improve efficiency for example, however this study is left for future research.
\begin{figure}
 \centering
 \includegraphics[width=14cm]{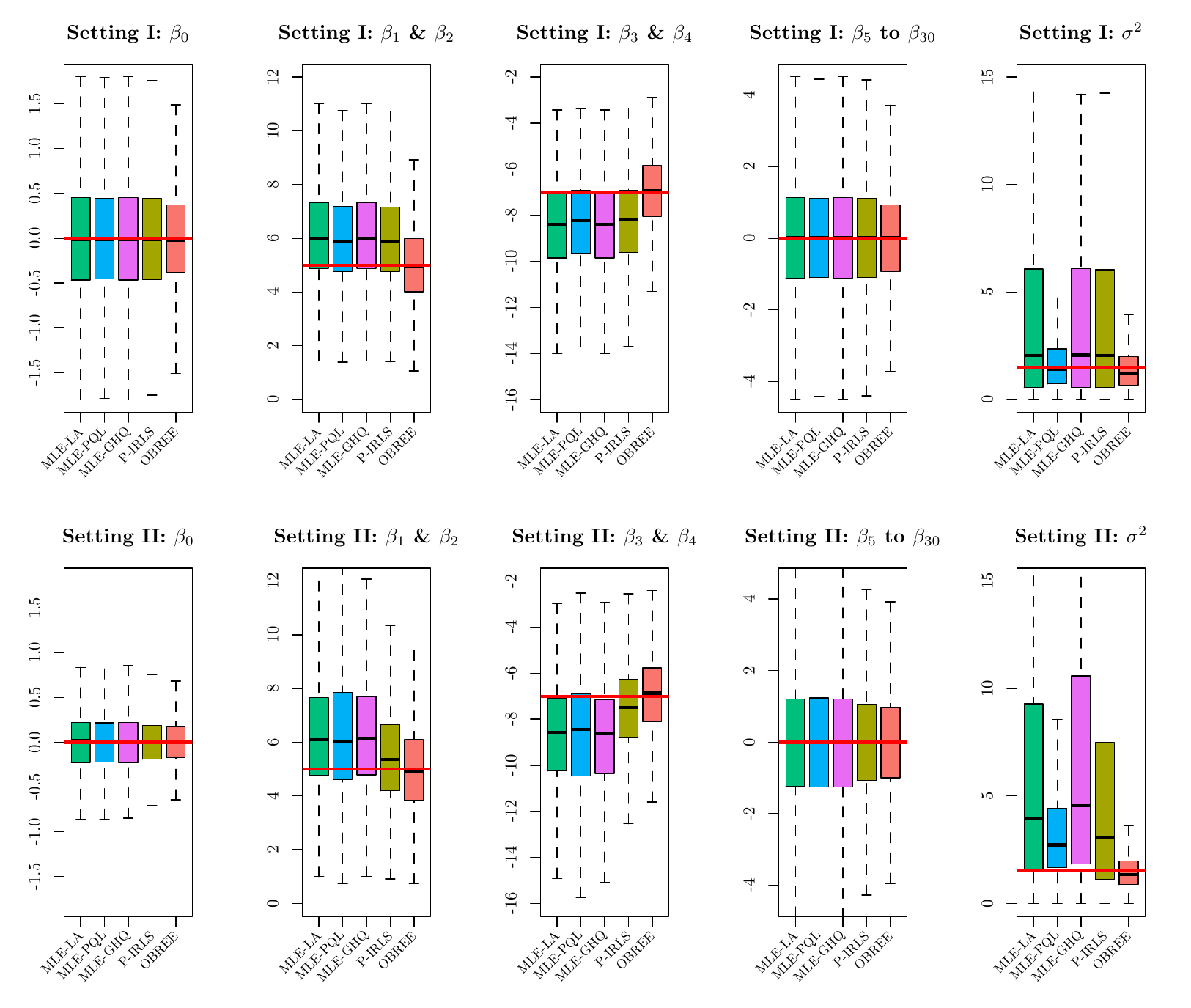}
  \caption{Finite sample distribution of estimators for the logistic regression with a random intercept, using the simulation settings presented in Table \ref{tab:sim-logistic2}. The estimators are the MLE with Laplace approximation (MLE-LA), the MLE with adaptive Gauss-Hermite quadratures (MLE-GHQ), the PQL (MLE-PQL), the penalized iteratively reweighted least squares (P-IRLS) and the OBREE with intital estimator based on the P-IRLS (OBREE). For each simulation setting, $1000$ samples are generated.}
 \label{fig:sim-glmm-boxplot}
\end{figure}

\begin{figure}
 \centering
 \includegraphics[width=14cm]{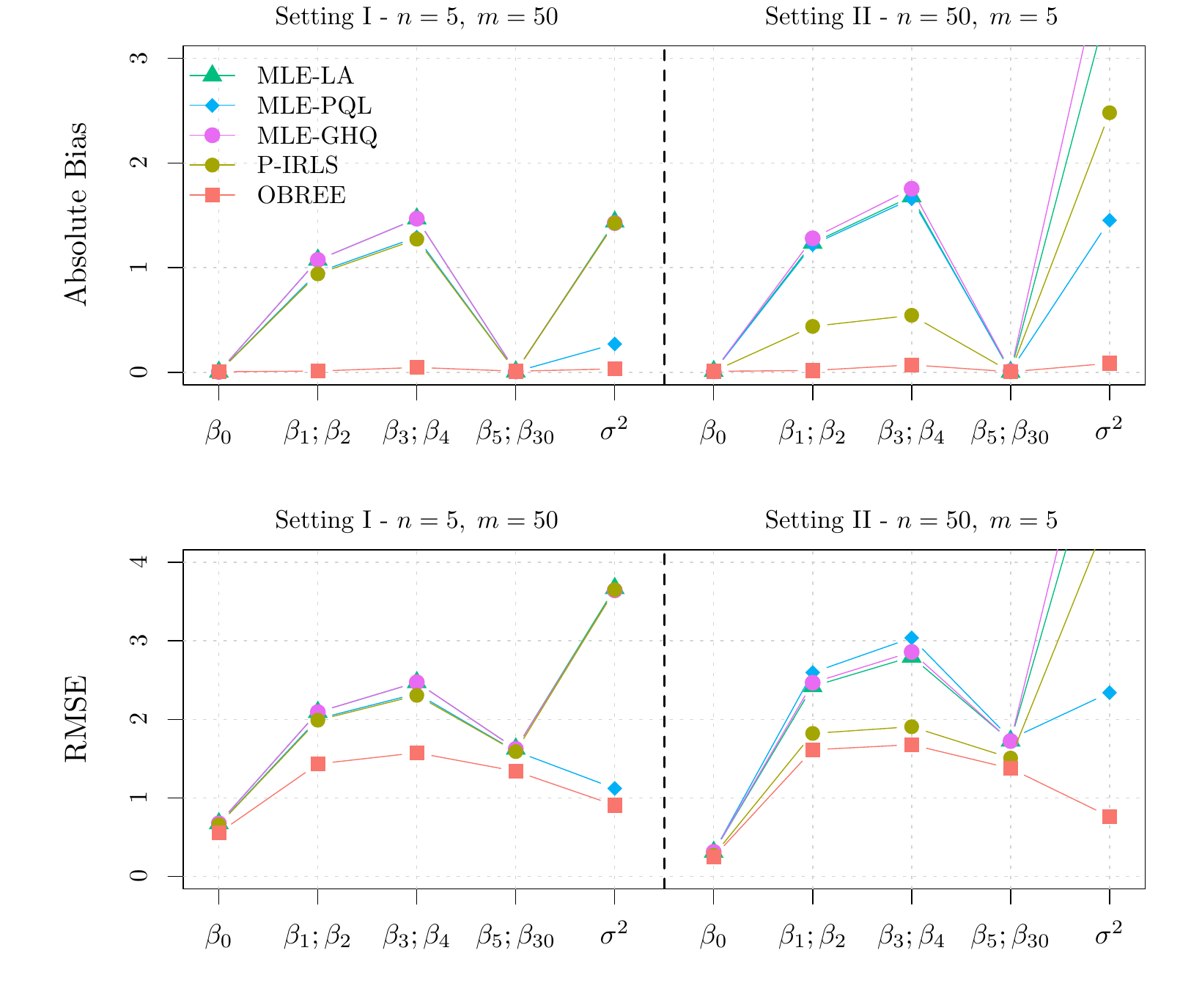}
  \caption{Finite sample bias and RMSE of estimators for the logistic regression with a random intercept, using the simulation settings presented in Table \ref{tab:sim-logistic2}. TThe estimators are the MLE with Laplace approximation (MLE-LA), the MLE with adaptive Gauss-Hermite quadratures (MLE-GHQ), the PQL (MLE-PQL), the penalized iteratively reweighted least squares (P-IRLS) and the OBREE with intital estimator based on the P-IRLS (OBREE). For each simulation setting, $1000$ samples are generated.}
 \label{fig:sim-glmm-summary}
\end{figure}

\newpage

\appendix
\titleformat{\section}[hang]{\large\center\scshape}{{\sc Appendix} \Alph{section}}{1em}{}

\section{Proof of Theorem \ref{thm:bias}}\label{app:unbias:IB}
In this appendix, we provide the proof of Theorem~\ref{thm:bias}. We also state and prove Corollary \ref{coro:thm:bias} which guarantees the optimal bias reduction property when the asymptotic bias is sub-linear.

Before proving Theorem \ref{thm:bias}, we introduce a particular asymptotic notation and state a lemma that provides a strategy for proving the optimal bias reduction property. Let $f(n)$ and $g(n)$ be real-valued functions with $g(n)$ being strictly positive for all $n \in \N^\ast\vcentcolon= \N\backslash\left\{0\right\}$. We write $f(n)=\mathcal{O}_{\delta\in\N^\ast}\left\{g(n)^\delta\right\}$ if and only if :
\begin{equation}
\label{def:newO}
\exists n' > 0, \;\; \exists M > 0  \;\; \text{such that} \;\; \forall \delta \in \N^\ast,  \;\; \text{and }\;\; \forall n \geq n', \;\; \lvert f(n)\rvert \leq M^\delta g(n)^\delta.
\end{equation}
Lemma~\ref{lemma:newO} and the following discussion below show how this particular notation is relevant in order to demonstrate Theorem~\ref{thm:bias}.
\begin{Lemma}
\label{lemma:newO}
Let $g(n)$ be a strictly positive real-valued function such that 
$\displaystyle\lim_{n\to\infty} g(n) = 0$.
If $f(n)=\mathcal{O}_{\delta \in \N^\ast}\left\{g(n)^\delta\right\}$,
then there exists a $n^\ast \in \N^\ast$ such that $f(n) = 0$ for all $n \geq n^\ast$.
\end{Lemma}
\begin{proof}
Since $\displaystyle\lim_{n\to\infty}g(n)=0$, there exists a $n^\ast \geq n'$ such that $Mg(n^\ast) < 1$. Without loss of generality we can suppose that the function $g(n)$ is decreasing\footnote{Indeed, we can define a decreasing step function $\tilde{g}(n)$ such that $\displaystyle\lim_{n\to 0}\tilde{g}(n) = 0$ and $g(n) \leq \tilde{g}(n)$, for all $n \in \N^\ast$.}. Therefore, for all $n \geq n^\ast$ we have $Mg(n) < 1$ and hence, 
\begin{equation*}
    \lvert f(n) \rvert \leq M^\delta g(n)^\delta \xrightarrow[\delta\to\infty]{}0.
\end{equation*}
In other words, $f(n)=0$ for all $n\geq n^\ast$ which ends the proof.
\end{proof}
he condition of Lemma~\ref{lemma:newO}, namely  $f(n)=\mathcal{O}_{\delta \in \N^\ast}\{g(n)^\delta\}$, may seem very strong. One may think for example that $f(n)=\mathcal{O}\{g(n)^\delta\}$ for all $\delta\in\N^\ast$ is sufficient. However, the exponential function provides a counter-example. Indeed, since $\displaystyle\lim_{n\to\infty} \exp(-n)\;n^{\delta}  = 0$ for all $\delta\in\N^\ast$, we have that  $\exp(-n)=\mathcal{O}(n^{-\delta})$ for all $\delta\in\N^\ast$, which means that
\begin{equation}\label{eq:exp:discuss}
  \forall \delta \in \N^\ast, \;\; \exists n_\delta > 0, \;\; \exists M_\delta > 0  \;\; \text{such that} \;\; \forall n \geq n_\delta, \;\; \exp(-n) \leq M_\delta n^{-\delta}.
\end{equation}
Since $\exp(-n)\neq0$ for all $n\in\N^\ast$, $f(n)=\mathcal{O}\{g(n)^\delta\}$ for all $\delta\in\N^\ast$ appears clearly to be not sufficient for reaching the conclusion of Lemma~\ref{lemma:newO}. 
One may attribute this failure to the dependence in $n_\delta$ in~\eqref{eq:exp:discuss} and propose the following stronger condition
\begin{equation}\label{eq:exp:discuss:1}
    \exists n' > 0, \;\; \forall \delta \in \N^\ast, \;\;  \exists M_\delta > 0  \;\; \text{such that} \;\; \forall n \geq n' \;\;  \lvert f(n) \rvert \leq M_\delta n^{-\delta}.
\end{equation}
Once again, the exponential function $\exp(-n)$ constitutes a counter-example. Indeed, setting $n'\vcentcolon=1$ and $M_\delta \vcentcolon= \delta !$, we have, for all $n \geq 1$ and all $\delta \in \N^\ast$, 
\begin{equation*}
    \exp(n)=\sum_{k=0}^{\infty}\frac{n^k}{k!}>\frac{n^\delta}{\delta!}\quad\Longrightarrow\quad\exp(-n)<\delta!\; n^{-\delta},
\end{equation*}
implying that $\exp(-n)$ satisfies \eqref{eq:exp:discuss:1}. However, we still have that $\exp(-n)\neq0$ for all $n\in\N^\ast$ and thus the conclusion of Lemma~\ref{lemma:newO} cannot be reached. Both of these examples suggest that the stronger condition in~\eqref{def:newO} is indeed necessary.

\begin{proof}[Proof of Theorem \ref{thm:bias}]
    By definition in~\eqref{eq:indirectInf:hTimesN}, we have 
	\begin{equation*}
	    \tilde{\bt} =  \bm{\pi}^*(\hat{\bt}, n).
	\end{equation*}
	Using Assumption~\ref{assum:B}, we re-express each side of the above equation as follows:
	\begin{equation*}
	    \begin{aligned}
	     \tilde{\bt} &= \bt_0 + \mathbf{b} (\bt_0, n) + \mathbf{v} \left(\bt_0, n\right)\\
	     \bm{\pi}^*(\hat{\bt}, n) =  \frac{1}{H} \sum_{h = 1}^H \tilde{\bt}^{*}_{h} &= \hat{\bt} + \mathbf{b} (\hat{\bt}, n) +  \frac{1}{H} \sum_{h = 1}^H  \mathbf{v}^*_h(\hat{\bt}, n),
	    \end{aligned}
	\end{equation*}
	where for all $h = 1, \dots , H$, $\mathbf{v}^*_h(\hat{\bt}, n) \vcentcolon= \tilde{\bt}^{*}_{h} - \bm{\pi}(\hat{\bt}, n)$ is a zero-mean random vector.
	We directly obtain
	\begin{equation*}
	    \mathbf{0} = \mathbb{E}\left[\bm{\pi}^*(\hat{\bt}, n) - \tilde{\bt} \right] = \mathbb{E}\left[\hat{\bt} - \bt_0\right] +\mathbb{E} \left[\mathbf{b} (\hat{\bt}, n) - \mathbf{b} (\bt_0, n)\right], 
	\end{equation*}
	which yields
	\begin{equation}
	    \mathbb{E}\left[\hat{\bt} - \bt_0\right] = - \mathbb{E} \left[\mathbf{b} (\hat{\bt}, n) - \mathbf{b} (\bt_0, n)\right].
	    \label{eq:proof:bias:inter:2}
	\end{equation}
    By Assumptions~\ref{assum:A} and~\ref{assum:B}, $\mathbf{b}(\hat{\bt},n)$ is a bounded random variable on a compact set and $\mathbf{b}(\bt,n) = \mathcal{O}(n^{-\beta})$ elementwise. We thus have 
    \begin{equation*}
        \mathbb{E} \left[\mathbf{b} (\hat{\bt}, n) - \mathbf{b} (\bt_0,n) \right] = \mathcal{O}\left(n^{-\beta}\right),
    \end{equation*} 
    elementwise. Consequently, we deduce from \eqref{eq:proof:bias:inter:2} that 
    \begin{equation}
      \lVert \mathbb{E} \left[ \hat{\bt} - \bt_0 \right]\rVert_\infty = \mathcal{O}\left(n^{-\beta}\right).
        \label{eq:proof:bias:inter:3}
    \end{equation}
    %

    The main idea is to re-evaluate $\mathbb{E}\left[\mathbf{b}(\hat{\bt},n)-\mathbf{b} (\bt_0, n)\right]$ using the mean value theorem as it allows to demonstrate by induction that, for all $\delta\in\N^\ast$, 
	\begin{equation*}
	     \Big\lVert\mathbb{E}\left[\hat{\bt}-\bt_0 \right]\Big\rVert_2 = \mathcal{O}\left(p^{\nicefrac{1}{2}}n^{-\delta\beta}\right).
	\end{equation*}
	Then, few extra steps enables to satisfy Lemma~\ref{lemma:newO} which concludes the proof.
	
    Applying the mean value theorem for vector-valued functions to $\mathbf{b}(\hat{\bt},n)-\mathbf{b}(\bt_0,n)$ we have
	\begin{equation*}
	    \mathbf{b}(\hat{\bt},n)-\mathbf{b}(\bt_0, n) = \mathbf{B}\left(\bt^{(\mathbf{b})},n\right) \left(\hat{\bt}-\bt_0\right).
	\end{equation*}
	where $\mathbf{B}\left(\bt,n\right)\vcentcolon =\frac{\partial}{\partial\,\bt^T} \mathbf{b}\left(\bt,n\right)\in\real^{p\times p}$ and $\bt^{(\mathbf{b})}$ corresponds to a set of $p$ vectors lying in the segment $(1 - \lambda) \hat{\bt} + \lambda {\bt}_{0}$ for $0\leq\lambda\leq 1$ (with respect to the function $\mathbf{b}(\bt,n)$). 
	By Assumptions \ref{assum:A} and \ref{assum:B}, $\mathbf{B}(\bt^{(\mathbf{b})},n)$ is also a bounded random variable. Moreover, by Assumption \ref{assum:B} again,  $\mathbf{B}(\bt^{(\mathbf{b})},n) = \mathcal{O}(p^{-1}n^{-\beta})$ elementwise since $\mathbf{b}(\bt,n) = \mathcal{O}(n^{-\beta})$ elementwise and 
	\begin{equation*}
	    \mathbf{b} (\bt, n) = \mathbf{b} (\bt_0, n) + \mathbf{B}\left(\bt^{(\mathbf{b})},n\right) \left(\bt - \bt_0 \right),
	\end{equation*}
	for all $\bt \in \bm\Theta$.
	For simplicity, we denote 
	\begin{equation*}
	    \mathbf{B} \vcentcolon= \mathbf{B
	    }\left(\bt^{(\mathbf{b})},n\right),\ \ \
	    \bm{\Delta}^{(\mathbf{b})} \vcentcolon=  \mathbf{b} (\hat{\bt}, n) - \mathbf{b} (\bt_0, n)
	    \ \ \ \text{and} \ \ \ \bm{\Delta} \vcentcolon= \hat{\bt} - \bt_0,
	\end{equation*} 
	which implies that $\bm{\Delta}^{(\mathbf{b})}=\mathbf{B}\bm{\Delta}$. Moreover, a consequence of \eqref{eq:proof:bias:inter:3} is that
	\begin{equation*}
	    \mathbb{E}\left[\bm{\Delta}_l\right]=\mathcal{O}\left(n^{-\beta}\right),
	\end{equation*}
	for any $l=1,\dots,p$ and hence 
	\begin{equation*}
	    \Vert\mathbb{E}\left[\bm{\Delta}\right] \Vert_2 = \mathcal{O}\left(p^{\nicefrac{1}{2}}n^{-\beta}\right).
	\end{equation*}
	Now, we have 
	\begin{equation}
	    \bm{\Delta}^{(\mathbf{b})}_l=\sum^{p}_{m=1}\mathbf{B}_{l,m}\bm{\Delta}_m \leq p \max_{m}\mathbf{B}_{l,m}\bm{\Delta}_m.
	    \label{proof:bias:eq:inter4}
	\end{equation}
    Using Cauchy-Schwarz inequality, we have 
    \begin{equation*}
      \mathbb{E}\left[\lvert\mathbf{B}_{l,m}\bm{\Delta}_m\rvert\right] \leq  \mathbb{E}\left[\mathbf{B}_{l,m}^2\right]^{\nicefrac{1}{2}} \mathbb{E}\left[\bm{\Delta}_m^2\right]^{\nicefrac{1}{2}}.
    \end{equation*}
	Since $\mathbf{B}_{l,m}$ and $\bm{\Delta}_m$ are bounded random variables on a compact set, $\mathbb{E}[\mathbf{B}_{l,m}]=\mathcal{O}(p^{-1}n^{-\beta})$ and $\mathbb{E}[\bm{\Delta}_m]=\mathcal{O}(n^{-\beta})$, so we have
	\begin{equation*}
        \mathbb{E} \left[ \mathbf{B}_{l,m}^2 \right]^{\nicefrac{1}{2}} \mathbb{E} \left[ \bm{\Delta}_m^2 \right]^{\nicefrac{1}{2}} =\mathcal{O}\left( p^{-1}n^{-2\beta} \right).
	\end{equation*}
	Therefore, we obtain 
	\begin{equation*}
    \mathbb{E}\left[\left\lvert\mathbf{B}_{l,m}\bm{\Delta}_m\right\rvert\right] = \mathcal{O}\left( p^{-1}n^{-2\beta} \right),
    \end{equation*}
	and hence, using \eqref{proof:bias:eq:inter4} we deduce that
	\begin{equation*}
      \mathbb{E}\left[\bm{\Delta}^{(\mathbf{b})}_l\right] \leq p\mathbb{E} \left[ \max_{m}\left\lvert\mathbf{B}_{l,m}\bm{\Delta}_m\right\rvert\right]=\mathcal{O}\left(n^{-2\beta}\right). 
    \end{equation*}
    Since $\mathbb{E}\left[\bm{\Delta}\right] = - \mathbb{E}\left[\bm{\Delta}^{(\mathbf{b})}\right]$, we have
	\begin{equation*}
	    \mathbb{E}\left[\bm{\Delta}_l\right]=\mathcal{O}\left(n^{-2\beta}\right),
	\end{equation*}
	and consequently,
	\begin{equation*}
	   \left\lVert\mathbb{E}\left[\bm{\Delta}\right]\right\rVert_2 =  \mathcal{O}\left(p^{\nicefrac{1}{2}}n^{-2\beta}\right).
	\end{equation*}
	Since $\mathbb{E}[\bm{\Delta}]=-\mathbb{E}[\bm{\Delta}^{(\mathbf{b})}]$ and $\mathbb{E}[\bm{\Delta}^{(\mathbf{b})}] = \mathbb{E}[\mathbf{B}\bm{\Delta}]$, one can repeat the same computations and deduce by induction that, for all $\delta\in\N^\ast$,
	\begin{equation*}
	    \Vert \mathbb{E}\left[\bm{\Delta}\right] \Vert_2 = \mathcal{O}\left(p^{\nicefrac{1}{2}}n^{-\delta\beta}\right).
	\end{equation*} 
	In order to show that 
	\begin{equation}\label{eq:unbias:perfect}
	    \Vert \mathbb{E}\left[\bm{\Delta}\right] \Vert_2 = \mathcal{O}_{\delta\in\N^\ast}\left(p^{\nicefrac{1}{2}}n^{-\delta\beta}\right),
	\end{equation}
	which implies $\Vert\mathbb{E}[\bm{\Delta}]\Vert_2= \mathcal{O}_{\delta\in\N^\ast}\left\{(p^{\nicefrac{1}{2}}n^{-\beta})^{\delta}\right\}$, it is sufficient to demonstrate that     
	\begin{equation}\label{eq:unbias:perfect:1}
	    \Vert \mathbb{E}\left[\bm{\Delta}\right] \Vert_\infty = \mathcal{O}_{\delta\in\N^\ast}\left(n^{-\delta\beta}\right).
	\end{equation}   
	Since for any $l,m=1,\cdots,p$, we have $\mathbb{E}[\lvert\bm{\Delta}_l\rvert]=\mathcal{O}(n^{-\beta})$ and $\mathbb{E}[\lvert\mathbf{B}_{l,m}\rvert]=\mathcal{O}(p^{-1}n^{-\beta})$, we obtain that 
	\begin{equation*}
	    \exists n_{\Delta} > 0, \;\; \exists M_{\Delta} > 0 \;\; \text{such that} \;\; \forall n \geq n_{\Delta}, \;\; \mathbb{E}\left[\lvert\bm{\Delta}_l\rvert\right] \leq M_{\Delta} n^{-\beta}
     \end{equation*}
    and
  	\begin{equation*}
        \exists n_{B} > 0, \;\; \exists M_{B} > 0 \;\; \text{such that} \;\; \forall n \geq n_{B}, \;\; \mathbb{E}\left[\lvert\mathbf{B}_{l,m}\rvert\right] \leq M_{\Delta} n^{-\beta}.
    \end{equation*}
	Without loss of generality, we suppose that $M_\Delta\geq1$ and $M_B\geq1$. Therefore, setting $n'\vcentcolon=\max\left(n_\Delta,n_B\right)$ and $M\vcentcolon=M_\Delta M_B$, our previous computations implies that, for all $\delta\in\N^\ast$ and all $n\geq n'$, we have
	\begin{equation*}
	   \lVert\mathbb{E}\left[\bm{\Delta}\right]\rVert_\infty \leq M^{\delta}n^{-\delta\beta}.
	\end{equation*}
	Hence \eqref{eq:unbias:perfect:1} holds true which ends the proof.
\end{proof}
	
Interestingly, the optimal asymptotic bias reduction property holds true in settings that are (way) more extreme than required by Assumption \ref{assum:B}, namely $p = o(n^{2\beta})$. Indeed, according to \eqref{eq:unbias:perfect}, the optimal asymptotic bias reduction property holds as long as $p^{\nicefrac{1}{2}} =~o(n^{\gamma})$ for some $\gamma \in  \N$, and in particular, even when $\gamma > 2\beta$. However, the condition $p = o(n^{2\beta})$ is used to prove the consistency of $\hat{\bt}$ thus justifying Assumption \ref{assum:B} as it is stated for readability purposes.

As previously said, the optimal bias reduction property of $\hat{\bt}$ can also be achieved when the initial estimator $\tilde{\bt}$ has a sub-linear asymptotic bias, i.e. $\tilde{\bt}$ can be written as follows,
\begin{equation}\label{eq:initial:estim:with:bias}
    \tilde{\bt}=\bt_0 + \mathbf{a}(\bt_0) +\mathbf{b}(\bt_0,n) + \mathbf{v}(\bt_0,n),
\end{equation}
where $\mathbf{a}(\bt_0)\vcentcolon=\mathbf{A}\bt_0+\mathbf{c}$ with $\mathbf{A}\in\real^{p\times p}$ and $\mathbf{c}\in\real^{p}$. In this case, the rate at which $p$ is allowed to grow with $n$ is more~restrictive.
\begin{Corollary}\label{coro:thm:bias}
Under Assumptions~\ref{assum:A},~\ref{assum:B} and~\ref{assum:C} and assuming that: (i) $\tilde{\bt}$ has the decomposition given in \eqref{eq:initial:estim:with:bias};  (ii) $\left(\mathbf{I} + \mathbf{A}\right)^{-1}$ exists; and (iii) there exists $\varepsilon>0$ such that  $p^{1+\varepsilon} = o(n^{\beta})$; we have $\big\lVert\mathbb{E}[\hat{\bt}]-\bt_0\big\rVert_2=\mathcal{O}(0)$.
\end{Corollary}
Even if the optimal bias reduction can be achieved when the asymptotic bias is large (there are no constraints on the norms of $\mathbf{A}$ and $\mathbf{c}$), in practice, the sub-linear assumption is more likely to be (nearly) satisfied when the asymptotic bias small as illustrated by our simulation studies (see Section \ref{sec:applica}).
\begin{proof}
Using the same reasoning as in the proof of Theorem \ref{thm:bias}, equality \eqref{eq:proof:bias:inter:2} becomes, in this case,
\begin{equation}
    \mathbb{E}\left[\hat{\bt} - \bt_0\right] = - \left(\mathbf{I} + \mathbf{A}\right)^{-1}\mathbb{E} \left[\mathbf{b} (\hat{\bt}, n) - \mathbf{b} (\bt_0, n)\right],
\end{equation}
which implies that 
\begin{equation}
    \mathbb{E} \left[\hat{\bt} - \bt_0 \right] = \mathcal{O}\left(pn^{-\beta}\right),
\end{equation}
elementwise.
The rest of the argument is similar to the one in the proof of Theorem \ref{thm:bias} and leads to 	
\begin{equation*}
	    \big\Vert \mathbb{E}[\hat{\bt} - \bt_0] \big\Vert_2 = \mathcal{O}_{\delta\in\N^\ast}\left\{\left(p^{\frac{\delta+1}{\delta}}n^{-\beta}\right)^{\delta}\right\}.
\end{equation*}
Since $\frac{\delta+1}{\delta}\to 1$, when $\delta\to\infty$, the previous equality implies, by the last assumption of Corollary \ref{coro:thm:bias}, that 
 $   \big\Vert\mathbb{E}[\hat{\bt}-\bt_0]\big\Vert_2 = \mathcal{O}_{\delta\in\N^\ast}\left\{\left(p^{1+\varepsilon}n^{-\beta}\right)^{\delta}\right\}$,
which ends the proof by Lemma~\ref{lemma:newO}.
\end{proof}

\section{Proof of Proposition \ref{thm:stat_prop}}\label{app:stat_prop}

The proof of Proposition~\ref{thm:stat_prop} is split into two lemmas. The first concerns the consistency of the OBREE $\hat{\bt}$ and the second its asymptotic normality.
\begin{Lemma}\label{lemma:consist}
Under Assumptions~\ref{assum:A},~\ref{assum:B} and~\ref{assum:C}, the OBREE is such that 
\begin{equation*}
\big\lVert \hat{\bt} - \bt_0 \big\rVert_2
    = o_{\rm p}(1).
\end{equation*}
\end{Lemma}
\begin{proof}
This proof is directly obtained by verifying the conditions of Theorem 2.1 of \citet{newey1994large} on the functions $Q(\bt)$ and $\widehat{Q}(\bt, n)$ defined as follow:
\begin{equation*}
    Q(\bt)\vcentcolon=\big\Vert\bt_0-\bt\big\Vert_2,\qquad\widehat{Q}(\bt,n)\vcentcolon= \big\Vert\tilde{\bt}-\bm{\pi}^\ast(\bt, n)\big\Vert_2,
\end{equation*}
where $\tilde{\bt}=\bm{\pi}(\bt_0,n) + \mathbf{v}(\bt_0,n) = \bt_0 + \mathbf{b}(\bt_0,n)+\mathbf{v}(\bt_0,n)$ and $\bm{\pi}^\ast(\bt, n) = \frac{1}{H} \sum_{h = 1}^H \tilde{\bt}^\ast_h$. 
Reformulating the requirements of this theorem to our setting, we have to show that  (i) $\bt$ is compact, (ii) ${Q}(\bt)$ is continuous, (iii) ${Q}(\bt)$ is uniquely minimized at $\bt_0$, (iv) $\widehat{Q}(\bt, n)$ converges uniformly in probability to $Q(\bt)$.

On the one hand, Assumption~\ref{assum:A} ensures that $\bt$ is compact. On the other hand, ${Q}(\bt)$ is trivially continuous and uniquely minimized at $\bt_0$. What remains to be shown is that $\widehat{Q}(\bt, n)$ converges uniformly in probability to $Q(\bt)$, which is equivalent to show that: for all $ \varepsilon > 0$ and for all $ \delta > 0$, there exists a sample size $n^\ast\in\N^\ast$ such that for all $n \geq n^\ast$
\begin{equation*}
    \Pr\left\{\sup_{\bt\in\bm\Theta}\;\Big\lvert\widehat{Q}(\bt, n)-Q(\bt)\Big\rvert\geq \varepsilon\right\}\leq\delta.
\end{equation*}
Fix $\varepsilon > 0$ and $\delta > 0$. Using the above definitions, we have that
\begin{equation}
    \sup_{\bt \in \bm\Theta}\;\Big\lvert\widehat{Q}(\bt,n)-Q(\bt)\Big\rvert\leq \sup_{\bt \in \bm\Theta} \; \left[\left\lvert\widehat{Q}(\bt,n)-Q(\bt,n)\right\rvert + \left\lvert Q(\bt, n)-Q(\bt) \right\rvert\right],
    \label{eq:convergen_proof_eq}
\end{equation}
where
\begin{equation*}
        {Q}(\bt, n)\vcentcolon = \big\Vert{\bm{\pi}}(\bt_0, n) - {\bm{\pi}}(\bt, n)\big\Vert_2.
\end{equation*}
Considering the first term on the right hand side of \eqref{eq:convergen_proof_eq}, we have
\begin{equation*}
  \begin{aligned}
      \Big\lvert\widehat{Q}(\bt, n) - Q(\bt, n) \Big\rvert 
      &\leq  \big\Vert \tilde{\bt} - \bm{\pi}^{*}\left(\bt, n\right) -   {\bm{\pi}}(\bt_0, n)  +  {\bm{\pi}}(\bt, n) \big\Vert_2 \\
      &\leq  \big\Vert \tilde{\bt} - {\bm{\pi}}(\bt_0, n) \big\Vert_2 + \big\Vert {\bm{\pi}}(\bt, n) - \bm{\pi}^{*}\left(\bt, n\right) \big\Vert_2 \\
      &= \big\Vert \mathbf{v} \left(\bt_0, n\right) \big\Vert_2 + \Big\Vert \frac{1}{H}\sum_{h=1}^H\mathbf{v}^{*}_h \left(\bt, n\right) \Big\Vert_2 \\
      &= \mathcal{O}_{\rm p} \left(\sqrt{p} n^{-\alpha} \right) + \mathcal{O}_{\rm p} \left(\sqrt{p} n^{-\alpha} H^{-\nicefrac{1}{2}}\right) = \mathcal{O}_{\rm p} \left(\sqrt{p} n^{-\alpha} \right),
  \end{aligned} 
\end{equation*}
by Assumption \ref{assum:C}. Similarly, we have
\begin{equation*}
  \begin{aligned}
    \Big\lvert{Q}(\bt, n) - Q(\bt) \Big\rvert &\leq \big\Vert {\bm{\pi}}(\bt_0, n) - {\bm{\pi}}(\bt, n) - \bt_0 +\bt \big\Vert_2  \\
      &=  \big\Vert \mathbf{b}(\bt_0, n) - \mathbf{b}(\bt, n)   \big\Vert_2 = \mathcal{O} \left(\sqrt{p}  n^{-\beta}\right),
  \end{aligned} 
\end{equation*}
by Assumption \ref{assum:B}. Therefore, we obtain
\begin{equation*}
    \sup_{\bt \in \bm\Theta} \; \Big\lvert\widehat{Q}(\bt,n) - Q(\bt)\Big\rvert = \mathcal{O}_{\rm p} \left(\sqrt{p} n^{-\alpha}\right) + \mathcal{O} \left(\sqrt{p}  n^{-\beta}\right).
\end{equation*}
By Assumptions \ref{assum:B} and \ref{assum:C}, there exists a sample size $n^\ast \in\N^\ast$ such that for all $n\in\N^*$ satisfying $n \geq n^\ast$ we have
\begin{equation*}
    \Pr \left\{ \sup_{\bt \in \bm\Theta} \; \Big\lvert\widehat{Q}(\bt, n)-Q(\bt) \Big\rvert\geq\varepsilon\right\}\leq\delta.
\end{equation*}
Therefore, the four condition of Theorem 2.1 of \citet{newey1994large} are verified implying the result.
\end{proof}
\begin{Lemma}\label{lemma:asym:norm}
Under Assumptions \ref{assum:A} to \ref{assum:D}, for any $\mathbf{s} \in \real^p$ such that $\lVert\mathbf{s}\rVert_2 = 1$ the OBREE satisfies
\begin{equation*}
    \sqrt{n} \mathbf{s}^{T}\left\{\left(1 + \frac{1}{H}\right)\bm{\Sigma}(\bt_0)\right\}^{-\nicefrac{1}{2}}\left(\hat{\bt} - \bt_0\right) \xrightarrow{\;d\;} \mathcal{N}\left(\mathbf{0}, 1\right).
\end{equation*}
\end{Lemma}
\begin{proof}
By definition, we have,
\begin{equation*}
    \begin{aligned}
        \hat{\bt} - \bt_0 = \mathbf{b}(\bt_0, n) - \mathbf{b}(\hat{\bt}, n) + \mathbf{v}(\bt_0, n) - \frac{1}{H}\sum^{H}_{h = 1}\mathbf{v}^{*}_h(\hat{\bt}, n).
    \end{aligned}
\end{equation*}
Setting $\mathbf{V}_{\bt_0} \vcentcolon=  \mathbf{s}^{T}
  \bm{\Sigma}(\bt_0)^{-\nicefrac{1}{2}}$, using Assumption \ref{assum:D} and the continuous mapping theorem, we have 
\begin{equation}\label{eq:lem:asym:norm:2}
    \sqrt{n}\mathbf{V}_{\bt_0} \left\{ \mathbf{b}(\bt_0, n) - \mathbf{b}(\hat{\bt}, n) \right\} = o_{\rm{p}}(1),
\end{equation}
since $\mathbf{b}(\bt,n)$ is continuous in $\bt$ by Assumption \ref{assum:B} and $\hat{\bt}$ is a consistent estimator of $\bt_0$ by Lemma \ref{lemma:consist}.
By Assumption \ref{assum:D} again, we have 
\begin{equation}\label{eq:lem:asym:norm:3}
  \sqrt{n}\mathbf{V}_{\bt_0} \mathbf{v}(\bt_0, n) \stackrel{d}{=} Z_0 + o_{\rm{p}}(1),
\end{equation}
where $Z_0$ is an independent $\mathcal{N}(0,1)$ random variable. Now, for any $h=1,\cdots, H$, we have
\begin{equation}\label{eq:lem:asym:norm:4}
  \sqrt{n}\mathbf{V}_{\bt_0} \mathbf{v}^{*}_h(\hat{\bt}, n)  \stackrel{d}{=} Z_h + o_{\rm{p}}(1),
\end{equation}
where $Z_h$ is an independent $\mathcal{N}(0,1)$ random variable. 
Since $\mathbf{V}_{\bt}$ is continuous in $\bt$ by Assumption~\ref{assum:D}, we have $\mathbf{V}_{\hat{\bt}} \xrightarrow{\;p\;}~\mathbf{V}_{\bt_0}$ by the continuous mapping theorem.
Therefore, combining~\eqref{eq:lem:asym:norm:2},~\eqref{eq:lem:asym:norm:3} and~\eqref{eq:lem:asym:norm:4}, we have by Slutsky's lemma
\begin{equation*}
\begin{aligned}
    &\left(1 + \frac{1}{H}\right)^{-\nicefrac{1}{2}}\sqrt{n}\mathbf{V}_{\bt_0} \left\{ \mathbf{v}(\bt_0, n) - \frac{1}{H}\sum^{H}_{h = 1}\mathbf{v}^{*}_h(\hat{\bt}, n) \right\} \\ 
    &\qquad\qquad \stackrel{d}{=}
     \left(1 + \frac{1}{H}\right)^{-\nicefrac{1}{2}}\sqrt{n} \left\{Z_0 - \frac{1}{H}\sum^{H}_{h = 1} Z_h + o_{\rm{p}}(1)\right\} \xrightarrow{\;d\;} \mathcal{N}\left(\mathbf{0}, 1\right),
    \end{aligned}
\end{equation*}
which ends the proof.
\end{proof}
\newpage
\section{Proof of Proposition \ref{prop:ib}}\label{app:ib}
\begin{proof}[Proof of Proposition~\ref{prop:ib}]
	We consider the function $T(\bt,n)$ defined in \eqref{eq:fixed:point} and recall that  
		\begin{equation*}
			T(\bt,n) = \bt + \tilde{\bt} - \bm{\pi}^{*}(\bt,n),
		\end{equation*}
		where $\bm{\pi}^{*}(\bt,n) = \frac{1}{H} \sum_{h = 1}^H \tilde{\bt}^{*}_h$. Formally, the function $T(\cdot,n)$ is defined on $\bm\Theta$ with target space $\real^p$ and is a deterministic function as the seeds used to compute $\tilde{\bt}^{*}_h$ (using the simulated samples $\mathbf{X}_h^*(\bt)$) are fixed. First, we show that $T(\cdot,n)$ is a contraction map for sufficiently large $n$. It enables us to apply Kirszbraun theorem (see \citealp{federer2014geometric}) and Banach fixed-point theorem to show that $T(\cdot,n)$ admits a unique fixed-point. 
		
		Let us consider $\bt_1,\bt_2\in\bm\Theta$, and compute 
		\begin{equation*}
			\begin{aligned}
				\left\lVert T(\bt_1,n) - T(\bt_2,n)\right\rVert_2^2 &= \left\lVert\mathbf{b}(\bt_2,n) - \mathbf{b}(\bt_1,n) + \frac{1}{H}\sum^{H}_{h=1} \mathbf{v}^{*}_h \left(\bt_2, n\right) - \mathbf{v}^{*}_h \left(\bt_1, n\right)  \right\rVert_2^2\\
				 &\leq \left\lVert\mathbf{b}(\bt_2,n) - \mathbf{b}(\bt_1,n) \right\rVert_2^2 + \frac{1}{H}\sum^{H}_{h=1}\left\lVert\mathbf{v}^{*}_h \left(\bt_2, n\right) - \mathbf{v}^{*}_h \left(\bt_1, n\right) \right\rVert_2^2 
				 \\
				 &= \left\lVert\mathbf{b}(\bt_2,n) - \mathbf{b}(\bt_1,n) \right\rVert_2^2 + \mathcal{O}_{\rm p}\left(pH^{-1}n^{-2\alpha}\right), 
			\end{aligned}
		\end{equation*}
		where the last equality is implied by Assumption \ref{assum:C}. Considering the first term of the last equality, by Assumption \ref{assum:B} and using the multivariate mean value theorem (with the same notation used in the proof of Theorem \ref{thm:bias}), we have
		\begin{equation*}
		   \left\lVert\mathbf{b}(\bt_2,n) - \mathbf{b}(\bt_1,n) \right\rVert_2^2 = \left\lVert\mathbf{B}(\bt_2 - \bt_1) \right\rVert_2^2 \leq \left\lVert\mathbf{B}\right\rVert_F^2 \left\lVert \bt_2 - \bt_1 \right\rVert_2^2,
		\end{equation*}
	    where $\lVert\cdot\rVert_F$ is the Frobenius norm. Using the same argument as in the proof of Theorem \ref{thm:bias} to compute the order of $\mathbf{B}$, we have 
	    \begin{equation*}
			\lVert\mathbf{B}\rVert_F=\sqrt{\sum_{j = 1}^p \sum_{l = 1}^p B_{j,l}^2}\leq p \max_{j,l = 1,\, \ldots, \, p} \lvert B_{j,l}\rvert = p\; \mathcal{O}\left(p^{-1}n^{-\beta}\right) =  \mathcal{O}\left(n^{-\beta}\right).
	    \end{equation*}
	    Therefore, we obtain 
	    \begin{equation*}
	        	\left\lVert T(\bt_1,n)-T(\bt_2,n)\right\rVert_2^2\leq \mathcal{O}\left(n^{-2\beta}\right) \left\lVert \bt_2 - \bt_1 \right\rVert_2^2 + \mathcal{O}_{\rm p}\left(pH^{-1}n^{-2\alpha}\right).
	    \end{equation*}
	    Since $\alpha, \beta > 0$ and by Assumption \ref{assum:C}, for sufficiently large $n$ we have that there exists $\varepsilon\in(0,1)$ such that for all $\bt_1,\,\bt_2\in\bm\Theta$ 
    	\begin{equation}
    	\label{T:contract}
		\left\Vert T(\bt_1,n) - T(\bt_2,n) \right\Vert_2 < \varepsilon \left\Vert\bt_2 - \bt_1\right\Vert_2.
	    \end{equation}
    	 Using Kirszbraun theorem, we can extend $T(\cdot, n)$ to a contraction map from $\real^p$ to itself. Therefore, applying Banach fixed-point theorem, there exists a unique fixed-point $\hat{\bt}\in\real^p$. However, by Assumption \ref{assum:B} and \ref{assum:C}, we have $\hat{\bt}\in\bm\Theta$ for large enough $n$.
    	 
		It remains to demonstrate that for all integer $k\geq0$, we have
		\begin{equation*}
         	\left\lVert\hat{\bt}^{(k)} - \hat{\bt}\right\rVert_2  =o_{\rm p}\left\{\exp(-k)\right\},
         \end{equation*}
		where the sequence $\left\{ \hat{\bt}^{(k)} \right\}_{k \in \N}$ is defined as in \eqref{eq:ib:seq}. This demonstration is straightforward by induction using \eqref{T:contract} and the fact that $\hat{\bt}^{(0)} = \tilde{\bt}$ and $\hat{\bt}$  (by Proposition \ref{thm:stat_prop}) are consistent estimator of $\bt_0$, that is 
		\begin{equation*}
		    	\left\lVert\hat{\bt}^{(0)} - \hat{\bt}\right\rVert_2 \leq \left\lVert\tilde{\bt} - \bt_0 \right\rVert_2 + \left\lVert\hat{\bt} - \bt_0 \right\rVert_2= o_{\rm p}\left(1\right).
		\end{equation*}
\end{proof}

\newpage
\bibliographystyle{agsm}
\bibliography{biblio}

\end{document}